\documentclass[a4paper, 12pt]{article}

\usepackage[utf8]{inputenc}
\usepackage[T1]{fontenc}
\usepackage[a4paper, margin=2.5cm]{geometry}
\usepackage{needspace}
\usepackage{bbm}

\usepackage{libertine} % text font
\usepackage{inconsolata} % texttt font

\usepackage{amsmath, amsthm, amssymb}
\usepackage{graphicx}
\usepackage{enumerate}
\usepackage{authblk}
\usepackage{thmtools}
\usepackage{thm-restate}
\usepackage[colorlinks=true, citecolor=red]{hyperref}
\usepackage[capitalize]{cleveref}
\usepackage{float}
\usepackage{amsfonts,mathtools,enumerate}
\usepackage[ruled,vlined]{algorithm2e}
\usepackage{url} %Added for adding url's to bibliography items

\usepackage{tikz}
\usetikzlibrary{arrows,decorations,backgrounds} % snakes

\renewenvironment{abstract}
{\small\vspace{-1em}
\begin{center}
\bfseries\abstractname\vspace{-.5em}\vspace{0pt}
\end{center}
\list{}{
\setlength{\leftmargin}{0.6in}%
\setlength{\rightmargin}{\leftmargin}}%
\item\relax}
{\endlist}

\declaretheorem[name=Theorem, numberwithin=section]{theorem}
\declaretheorem[name=Lemma, sibling=theorem]{lemma}

\declaretheorem[name=Definition, sibling=theorem]{definition}
\declaretheorem[name=Corollary, sibling=theorem]{corollary}

\declaretheorem[name=Claim, sibling=theorem]{claim}

\def\cqedsymbol{\ifmmode$\lrcorner$\else{\unskip\nobreak\hfil
\penalty50\hskip1em\null\nobreak\hfil$\lrcorner$
\parfillskip=0pt\finalhyphendemerits=0\endgraf}\fi}

\DeclareMathOperator{\polylog}{polylog}
\def\Pr{\mathbb{P}}
\def\Ex{\mathbb{E}}
\newcommand{\NN}{\mathbb{N}} % integers
\tikzstyle{vertex}=[circle, draw, fill=black!50,
                        inner sep=0pt, minimum width=4pt]

\let\leq\leqslant
\let\geq\geqslant

\newcommand{\qu}{\textsc{Query}}

\title{Distance reconstruction of sparse random graphs}

\author[1]{Paul Bastide}
\affil[1]{LaBRI - TU Delft, \href{mailto:paul.bastide@ens-rennes.fr}{paul.bastide@ens-rennes.fr}}
\date{\today}

\begin{document}
\maketitle
\begin{abstract}
    In the distance query model, we are given access to the vertex set of a $n$-vertex graph $G$, and an oracle that takes as input two vertices and returns the distance between these two vertices in $G$. We study how many queries are needed to reconstruct the edge set of $G$ when $G$ is sampled according to the $G(n,p)$ Erd\H{o}s-Renyi-Gilbert distribution.

    Our approach applies to a large spectrum of values for $p$ starting slightly above the connectivity threshold: $p \geq \frac{2000 \log n}{n}$. We show that there exists an algorithm that reconstructs $G \sim G(n,p)$ using $O( \Delta^2 n \log n )$ queries in expectation, where $\Delta$ is the expected average degree of $G$. In particular, for $p \in [\frac{2000 \log n}{n}, \frac{\log^2 n}{n}]$ the algorithm uses $O(n \log^5 n)$ queries. 
\end{abstract}

\section{Introduction}

%\begin{itemize}
    %\item Distance reconstruction general: introduction, why the low degree is interesting, conjecture Mathieu Kannan and Zhou.

    %\item Reconstruction in random graphs: Mathieu Zhou (RSA'23), Krivelevich Zhukovskii (ArXiv'24) (just cite here). Other settings (see RSA'23 for ref) (+ add metric dimension).

    %\item Recently Krivelevich Zhukovskii (ArXiv'24) [extensively did blablabla indepdently of this paper. They got tighter bound for dense graphs, we get something for sparse one.

    %\item Our result : This validates a statement of Mathieu and Zhou (RSA'23), Krivelevich and Zhukovskii (ArXiv'24). compare our result to there.
%\end{itemize}

The question of reconstructing a graph by querying a distance oracle is an extensively studied topic \cite{beerliova06,mathieu2013graph,RONG20221,mathieu2021simple,bastide2023optimal}. It has major applications in network discovery \cite{mathieu2013graph,bastide2023optimal,krivelevich2024reconstructing}, phylogenetics \cite{hein1989optimal,waterman1977additive,KingZhangZhou} and machine learning \cite{soudry2015efficient}. These kinds of problems arise naturally when confronted with the following type of questions: Given an unknown internet network, how fast can you reconstruct it only using a {\fontfamily{qcr}\selectfont ping} protocol? 
Where the {\fontfamily{qcr}\selectfont ping} protocol designates one of the most basic protocols available in an internet network, allowing you to test if another network node is available by sending a dummy message and receiving a dummy message back. In practice, it is possible to estimate the distance in the network by computing the delay of a {\fontfamily{qcr}\selectfont ping}, between the time the first message is sent and the time the second message is received. 

\paragraph{Distance reconstruction setup} Let $G=(V,E)$ be a connected $n$-vertex graph where $V$ is known but $E$ is hidden. We are given access to an oracle that takes $u,v \in V$ as input and answers $d_G(u,v)$ the distance between $u$ and $v$ in $G$. We study how many queries are needed to uniquely reconstruct the edge set $E$. This setting has been extensively studied both when the graph $G$ is assumed to be taken from a known graph class $\mathcal{G}$ \cite{mathieu2013graph,bastide2023optimal,rong21} and when $G$ is sampled randomly according to a known distribution  \cite{mathieu2021simple,krivelevich2024reconstructing}. The set of edges of a graph $G = (V,E)$ can be seen as $E = \{ \{u,v \} \in {V \choose 2} \mid d(u,v) = 1\}$. Therefore ${|V| \choose 2}$ is always an upper bound on the number of queries required to reconstruct a graph. Unfortunately, this upper bound cannot be improved for most classes containing graphs with a high maximum degree. For example, even when $\mathcal{G}$ is taken to be the class of all trees, Reyzin and Srivastava \cite{reyzin2007learning} showed that there is still a lower bound of $\Omega(n^2)$. However, when $\mathcal{G}$ is taken to be a class of graphs with bounded degree, \cite{mathieu2013graph} showed that there exists a randomized algorithm using only $O(n^{3/2})$ queries in expectation. Since then, most research has focused on bounded degree graph classes \cite{rong21,bastide2023optimal,mathieu2021simple}.

\paragraph{Reconstruction of random graphs} Mathieu and Zhou \cite{mathieu2021simple} initiated the study of the complexity of distance reconstruction on random graphs by showing that random $\Delta$-regular graphs (where $\Delta$ is a constant) can be reconstructed using only $O(n\log^2 n)$ queries in expectation. The algorithm they used is simple and natural, but the complexity analysis requires a profound understanding of how to sample a $\Delta$-regular graph. The authors mentioned that their algorithm could potentially lead to $O(n \polylog n)$ query upper bounds in different randomized settings, including $G(n,p)$ for values of $p$ close to the connectivity threshold.

Very recently, Krivelevich and Zhukovskii \cite{krivelevich2024reconstructing} studied the query complexity of reconstructing $G \sim G(n,p)$ for large values of $p \geq n^{-1+\varepsilon}$ and derived tight bounds for $p$ outside of some threshold points around which the diameter  
increases, explicitly: $p = n^{-\frac{k}{k+1} + o(1)}$ for $k \in \mathbb{N} \cup \{\infty\}$. In this range, they managed to pinpoint precisely the complexity to be $\Theta(n^{4-d}p^{2-d})$ queries with high probability, where $d$ is the diameter of $G$. For these values of $p$, the diameter is known to be a fixed constant independent of $n$ w.h.p.\footnote{with high probability: formally, with a probability that tends to $1$ when $n$ tends to infinity.}. They also studied this problem in the case of a non-adaptive algorithm\footnote{an algorithm where queries can be seen as simultaneous and do not depend on each other's answers.} and proved a bound of $\Theta(n^{4-d}p^{2-d} \log n)$ queries outside of the diameter increase threshold points mention above. This non-adaptive bound can be generalised for $p \gg \log^2n/n$, and they asked if their results could be further extended to smaller values of $p$.

\paragraph{Our contribution} In this paper, we focus on reconstructing $G \sim G(n,p)$  for $p = n^{-1 + o(1)}$, partially answering a question of \cite{krivelevich2024reconstructing}. Note that in this setting the average degree of $G$, often denoted $\Delta$ in this paper, is equal to $(1+o(1)) np$ with high probability.

\begin{restatable}{theorem}{main}
    \label{thm:main}
  For any $\varepsilon \geq 0$, and every $n \in \mathbb{N}$, for every $ \frac{2000\log n}{n} \leq p  \leq n^{-\frac12 - \varepsilon}$ , there exists an algorithm that reconstructs $G \sim G(n,p)$ using $O(\Delta^2 n\log n)$ queries in expectation, where $\Delta = (n-1)p$ is the expected average degree of $G$.
\end{restatable}

\cref{thm:main} covers a large continuous range of values of $p$. Our result compares to the results proved independently by Krivelevich and Zhukovskii \cite{krivelevich2024reconstructing}. When $p$ is close to the thresholds $p = n^{-\frac{k}{k+1}+o(1)}$ we obtain the same complexity, up to a $n^{o(1)}$ factor. However, Krivelevich and Zhukovskii proved tighter bounds when $p$ is away from this threshold and $k$ is a constant. In contrast to \cite{krivelevich2024reconstructing}, \cref{thm:main} addresses very small values of $p$, particularly those close to the connectivity threshold (around $\Theta(\log n/n)$). In this regime, its query complexity is only a factor $O(\Delta / \log \log n)$ away from the lower bound established by Kannan, Mathieu and Zhou \cite{kannan2018graph}, which is $\Omega(\Delta n \log n /\log \log n)$. \cref{thm:main} applied to $p \in [2000 \log n, \log^2 n]$ partially answers a question of Krivelevich and Zhukovskii \cite{krivelevich2024reconstructing} regarding extending their bounds outside the range $p \gg \log^2 n /n$. It also confirms the intuition of Mathieu and Zhou in \cite{mathieu2021simple} that their algorithm could potentially be applied to this range of $p$, albeit requiring non-trivial analysis.

\begin{corollary}
 For every $n \in \mathbb{N}$, for $p \in [\frac{2000 \log n}{n}, \frac{\log^2n}{n}]$ there exists an algorithm that reconstructs $G \sim G(n,p)$ using $O(n\log^{5} n)$ queries in expectation.
\end{corollary}

The author didn't try to optimize the constant $2000$ in the previous results. 

\medskip  

To prove \cref{thm:main}, we employ the natural algorithm introduced by Mathieu and Zhou \cite{mathieu2013graph} and also used in \cite{krivelevich2024reconstructing} as described in the preliminaries. Similar to previous works, the correctness of the algorithm is straightforward, with the main challenges lying in its complexity analysis. The primary novel aspect, compared to \cite{krivelevich2024reconstructing}, is deriving bounds even for very sparse graphs where the diameter grows with $n$ and $p \leq \log^2n/n$. This requires a detailed study of the interactions within the $k$-neighborhoods of non-adjacent vertices. More specifically, given two vertices $u,v \in V(G)$, and an integer $k$, we derive precise bounds for the number of vertices in the set $\{x \in V(G) \mid d(u,x) = k \, \wedge \, d(v,x) \in [k,k+1]\}$. Our method might also be applicable for establishing bounds for other types of distance profiles between pairs of vertices. That is given two vertices $u,v$ and two distances $i$ and $j$, computing tight bounds on the distribution of vertices satisfying $\{x \mid d(u,x) = i \wedge d(v,x) = j\}$.

\section{Preliminaries}
\label{sec:prelim}

We denote by $\qu_G(u,v)$ the query to the oracle that answers the shortest distance between the vertices $u$ and $v$ in $G$. We also slightly abuse notation and denote by $\qu_G(A,B)$ the set of queries $\{ \qu_G(a,b) \mid (a,b) \in A\times B\}$. We denote by $d_G$ the distance function of the graph $G$. We also denote by $N^k(u)$ (resp. $N^{\leq k}(u)$ the set of vertices at distance exactly $k$ (resp. at most $k$) from $u \in V(G)$. We often omit the subscript $G$ when $G$ is clear from context. For $G$ a graph and $A$ a subset of vertices of $G$ we denote by $G[A]$ the graph induced by $G$ on $A$. 
We also let $[k] = \{1,\ldots, k\}$ for $k \in \mathbb{N}$.

For $n \in \NN$ and $p \in [0,1]$, we denote by $\operatorname{Bin}(n,p)$ the binomial distribution of parameter $n$ and $p$ and by $\operatorname{B}(p)$ the Bernouilli distribution of parameter $p$. 

\subsection{A Simple Algorithm}

We sketch here the algorithm we are using, which has been introduced in \cite{mathieu2021simple} and also used in \cite{krivelevich2024reconstructing}. Consider a random set $S \subseteq V(G)$ of fixed size $|S|=s$. First, the algorithm asks $\qu(S,V(G))$. From these queries, it computes the \emph{pseudo-edges} of $G$ define as
$$
\Tilde{E} := \{ \{u,v\} \subseteq V(G) \mid \forall s \in S, |d(s,u) - d(s,v)| \leq 1 \}.
$$ 
The algorithm asks $\qu(\Tilde{E})$ and finishes. Note that this algorithm always completely reconstructs the graph. If there exists $s$ such that $|d(s,u) - d(s,v)| \geq 2$ then $uv \notin E(G)$, therefore $\Tilde{E}$ is a superset of $E(G)$. The number of queries asked is at most $ns + |\Tilde{E}|$.
Therefore the challenge in proving precise upper bound lies in showing that even for small values of $s$, $|\Tilde{E}|$ does not become excessively large.

\subsection{Tools}

\begin{lemma}[Chernoff's bound \cite{chernoff1952measure}]
    \label{chernoff}
   Given $n$ independent random i.i.d. variables $X_1,\ldots,X_n$ taking values in $\{0,1\}$, let us denote by $\mu = \Ex[\sum_{i=1}^n X_i]$. Then for any $0 < \delta < 1$,
   $$
   \Pr\left(\left|\sum_{i=1}^n X_i - \mu\right| \geq \delta \mu\right) \leq 2e^{-\delta^2 \mu/3}.
   $$
   Also for any $\delta \geq 0$,
   $$
   \Pr\left(\sum_{i=1}^n X_i \geq (1 +\delta) \mu\right) \leq e^{-\delta^2 \mu/(2 + \delta)}.
   $$
\end{lemma}

The following results are standard and can be derived from the first chapters of \cite{bollobas1998random}. We include proofs for completeness.

\begin{lemma}
    \label{lem:maxmindegree}
   For every $n \in \mathbb{N}$, for $p \geq \frac{2000 \log n}{n}$, let $G \sim G(n,p)$. Then, we have that $G$ has max degree at most $\frac32 p(n-1)$ and min degree at least $\frac12 p(n-1)$ with probability $1 - o(n^{-1})$.
\end{lemma}

\begin{proof}
    Fix a vertex $u \in V(G)$, the degree of $u$ is the sum of the independent $\{0,1\}$ variables $(X_v)_{v \in V(G)\setminus \{u\}}$ defined by $X_v = 1$ if and only if $uv \in E(G)$. Note that $X_v \sim B(p)$ for all $v$. We apply Chernoff's bound with $\delta = \frac12$ and $\mu = p(n-1)$.
    \begin{align*}
          \Pr\left(\left|\sum_{v \in V(G)\setminus \{u\}} X_v - p(n-1)\right| \geq \frac12 p(n-1) \right) \leq 2e^{-p(n-1)/12} \leq n^{-2},
    \end{align*}
    where the last inequality comes from $p \geq 2000 \log n /n$. It now suffices to take a union bound over all vertices to conclude.
\end{proof}

\begin{lemma}
    \label{lem:isolatedgnp}
    For $N \in \NN$, for any $\Delta > 0$, any $p \leq \frac{1}{N\Delta}$  and $G \sim G(N,p)$, the number of non-isolated vertices in $G$ is at most $4N/\Delta$ with probability at least $1 - 2e^{-N/(3\Delta)}$.
\end{lemma}

\begin{proof}
    Let $X_p$ denote the random variable counting the number of isolated vertices in $G(N,p)$. Note that $X_{\frac1{N\Delta}}$ stochastically dominates $X_p$ for $p \leq \frac1{N\Delta}$ (i.e for any $x \in \mathbb{R}$, $\Pr[X_{\frac1{N\Delta}} \geq x] \geq \Pr[X_{p} \geq x] $). Therefore we will only focus on $p = \frac{N}{\Delta}$ from now on.
    
    Note that an edge in $G$ creates at most two non-isolated vertices, therefore the number of non-isolated is at most twice the number of edges of $G$. Moreover we have $\Ex[|E(G)|] = p {N \choose 2} \leq N^2$. We bound the number of edges of $G$ using Chernoff's bound (\cref{chernoff}) with $\delta = 1$ and $\mu = p{N \choose 2}$.
    \begin{align*}
        \Pr\left( \left||E(G)| - \mu\right| \geq \frac{N} {\Delta} \right) \leq e^{-\frac{N}{3\Delta}}
    \end{align*}
    which, using the remark above, directly implies the lemma statement.
\end{proof}

\section{Proof of \cref{thm:main}}

From now on, and for the rest of the paper let us set $\frac{2000 \log n}{n} \leq p \leq n^{-1/2-\varepsilon}$. We consider $G \sim G(n,p)$. Let $\Delta = p(n-1)$. Note that $\Delta$ is the expected average degree in $G$. A key notion for the proof is the notion of \emph{witness} defined below. Intuitively, a witness $x$ of a non-edge $uv \notin E(G)$ is a vertex such the answers to $\qu(x,\{u,v\})$ are enough to conclude that $uv \notin E(G)$.

\begin{definition}
    We say that $s \in V(G)$ is a witness of the non-edge $uv \notin E(G)$ if $|d(s,u) - d(s,v)| \geq 2$. We denote $W_{uv}$ the set of witnesses of the non-edge $uv$.
\end{definition}

Let us first prove the following key lemma, stating that most pairs of vertices have a large number of witnesses.

\begin{lemma}[Main Lemma]
    \label{lem:linearwitnesses}
    The set $\mathcal{E} = \{uv \notin E(G) \mid W_{uv} \geq (1/2 - o(1)) n/\Delta^2\}$ has size at least ${n \choose 2} - O(n \Delta^2)$ with probability $1 - o(n^{-1})$.
\end{lemma}

\begin{proof}

    Since we aim to prove the statement with a probability of \(1 - o(n^{-1})\), we can apply a union bound to a constant number of events, each having a probability of \(1 - o(n^{-1})\). For example, we will assume deterministic conditioning on the fact that \cref{lem:maxmindegree} holds. We will also use similar reasoning throughout the proof, in which case, we will say that we assume that an event (here \cref{lem:maxmindegree}) holds deterministically.
    We first make the following easy claim:
    \begin{claim}
        \label{cl:dist3}
       There are at most $\frac94 n\Delta^2$ pairs of vertices $u,v \in V(G)$ satisfying $d(u,v) \leq 2$ with probability $1 - o(n^{-1})$.
    \end{claim}

    \begin{proof}
        For a fixed $u \in V(G)$, from \cref{lem:maxmindegree}, $|N^2(u)| \leq \left(\frac32 \Delta\right)^2  = \frac94 \Delta^2$. The claim follows from double counting.
    \end{proof}
    For the purpose of the lemma, we can put aside such pairs of vertices and only prove that the set $\mathcal{E}^*:= \{uv \notin E(G) \mid d(u,v) \geq 3 \wedge |W_{uv}| \geq n/\Delta^2\} \subseteq \mathcal{E}$ is large enough. We will in fact prove $\mathcal{E}^* = \{uv \notin E(G) \mid d(u,v) \geq 3 \}$ with probability $1 - o(n^{-1})$. To do so, let us consider a uniformly chosen pair of vertices $uv \notin E(G)$ satisfying $d(u,v) \geq 3$ and let us prove that $|W_{uv}| \geq (1/2 - o(1))n/\Delta^2$ with probability $1 - o(n^{-3})$.

    Let us denote by $N^k(u,v)$ the common $k$-neighbourhood of $u$ and $v$, 
    $$N^k(u,v) = \{x \in V(G) \mid \min(d(u,x),d(v,x)) = k \}.$$ 
    We will reveal $G$ by spheres of increasing radius around $u$ and $v$. By that, we mean that we first reveal $N(u,v)$ then $N^2(u,v), N^3(u,v), \ldots, N^k(u,v)$. We define iteratively a partition of $N^{k}(u,v)$ into two sets $A_k$ and $B_k$.
    Intuitively, $A_k$ will preserve properties that will ensure that most of the vertices in $A_k$ are witnesses of the pair $u,v$. To do so, we will put in $B_k$ any vertex that could potentially be a non-witness of $uv$. We divide this into three cases represented in \cref{fig:Bk}. We describe each part below and give a short informal explanation of why these vertices could be non-witnesses:
    \begin{itemize}
        \item $\mathcal{B}^k_1 = N(B_{k-1})$: Suppose $x \in B_{k-1}$ is a non-witness, for example $d(x,u) = d(x,v) = k-1$, then any neighbour $y$ of $x$ in $N^k(u,v)$ satisfies $d(y,u) = d(y,v) = d(x,u) + 1 = k$ and is a non-witness of $uv$, therefore $N(x)$ should be in $B_k$.
        
        \item $\mathcal{B}^k_2 = N(\{x \in A_{k-1} \mid N(x) \cap N^{k-1}(u,v) \neq \emptyset\})$: Consider $x,y \in A_{k-1}$ such that $xy \in E(G)$, and suppose that before the $xy$ edge is revealed, both $x$ and $y$ have a distance profile that makes them potential witnesses of $uv$. One possible case is $d_{G \setminus\{xy\}}(x,u) = d_{G \setminus\{xy\}}(y,v) = k-1$, $d_{G \setminus\{xy\}}(x,v) > k $ and $d_{G \setminus\{xy\}}(y,u) > k $. When the edge $xy$ is revealed the distance profiles are modified: $d(x,u) = d(y,v) = k-1$ and $d(x,v) = d(y,u) = k$. Both $x$ and $y$ become non-witnesses. Any neighbour $z \in N(x) \cup N^k(u,v)$ satisfy $d(z,u) = k$ and $k \leq d(z,v) \leq d(z,x) + d(x,v) = k+1$. A symmetric reasoning for $y$ implies that every vertex in $\left(N(x) \cup N(y)\right) \cap N^k(u,v)$ is not a witness of $uv$.
        
        \item $\mathcal{B}^k_3 = \left\{x \in N^k(u,v) \middle| |N(x) \cap A_{k-1}| \geq 2\right\}$: Suppose $x,y$ are two distinct witnesses of $uv$. Again, one possible case is $d(x,u) = d(y,u) = k-1$, $d(x,v) > k$ and $d(y,u) > k$. Suppose they share a neighbour $z$ in $N^k(u,v)$. Now the distance profile of $z$ is $d(z,u) = d(x,u) + 1 = k$ and $d(z,v) = d(y,v) + 1 = k$, therefore $z$ is a non-witness of $uv$.
    \end{itemize}

    Let $A_1 = N(u,v)$ and $B_1 = \emptyset$ we define recursively, for $k \in \NN$,
    \begin{align*}
        B_k &= N^k(u,v) \cap ( \mathcal{B}^k_1 \cup \mathcal{B}^k_2 \cup \mathcal{B}^k_3) \qquad A_k = N^k(u,v) \setminus B_k.
    \end{align*}

    \begin{figure}[H]
        \centering
        \begin{tikzpicture}[scale = 0.85]
\tikzstyle{vertex}=[circle, draw, fill=black!50,
                        inner sep=0pt, minimum width=4pt]
                        
\draw (-3,0) [color=black!35] ellipse (2.5cm and 1.5cm); 
\draw (3,0) [color=black!35] ellipse (2.5cm and 1.5cm); 
\draw (0,3) [color=black!35] ellipse (1cm and 0.75cm);
\draw (3,3) [color=black!35] ellipse (1cm and 0.75cm);
\draw (6,3) [color=black!35] ellipse (1cm and 0.75cm);
\draw (-3.5,3) [color=black!35] ellipse (2cm and 0.75cm);
\draw[dashed] (-8,1.85) -- (8,1.85);

\node[vertex] (a) at (4.5,0.5) {};
\node[vertex] (b) at (6,3) {};
\node[vertex] (c) at (3,3) {};
\node[vertex] (d) at (0,3) {};

\node[vertex] (e) at (-4.5,0.5) {};
\node[vertex] (f) at (-3.2,0.75) {};
\node[vertex] (f2) at (-2.5,-0.25) {};
\node[vertex] (g) at (-1.5,0.25) {};
\node[vertex] (h) at (-1,0) {};
\node[vertex] (i) at (-3.5,3) {};

\draw (a) -- (b);
\draw (h) -- (g) -- (c);
\draw (f) -- (d) -- (f2);
\draw (i) -- (e);

\node at (-7,0) {$N^{k-1}(u,v)$};
\node at (-7,3) {$N^{k}(u,v)$};
\node at (0,4.25) {$\mathcal{B}_3$};
\node at (3,4.25) {$\mathcal{B}_2$};
\node at (6,4.25) {$\mathcal{B}_1$};
\node at (-3.5,4.25) {$A_k$};
\node at (-3,-2) {$A_{k-1}$};
\node at (3,-2) {$B_{k-1}$};

\end{tikzpicture}
        \caption{Representation of the three parts $\mathcal{B}^k_1,\mathcal{B}^k_2,\mathcal{B}^k_3$ that compose $B_k$}
        \label{fig:Bk}
    \end{figure}
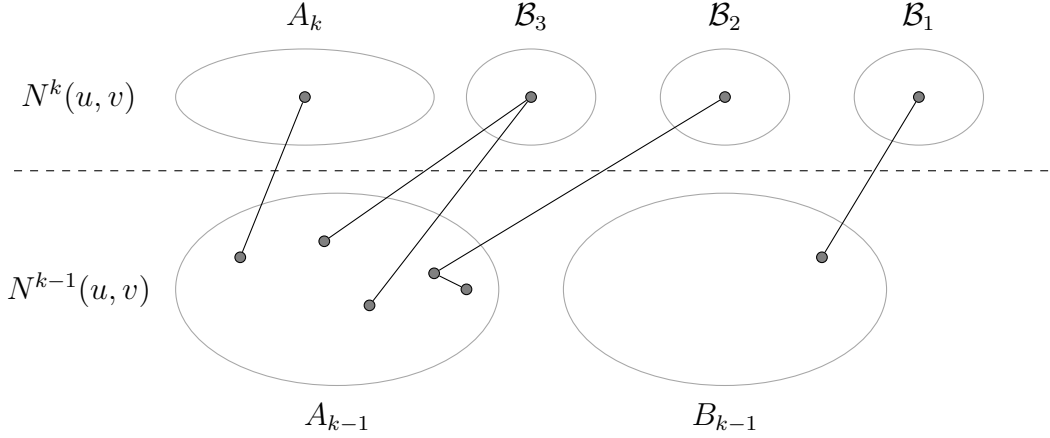

    The proof is structured around three claims. \cref{cl:Akisgood} ensures that at least half of the vertices of $A_k$ are witnesses of $u,v$ with high probability.

    \begin{restatable}{claim}{Akisgood}
        \label{cl:Akisgood}
        For any $k$ such that $|N^{k}(u,v)| \leq n/\Delta^2$, we have $|A_{k} \cap W_{uv}| \geq \frac12 |A_{k}|$ with probability $1 - o(n^{-3})$.
    \end{restatable}

    \cref{cl:initBksmall} and \cref{cl:Bkissmall} control the size of $B_k$ (and therefore $A_k$) inductively. Specifically, \cref{cl:initBksmall} addresses the case where $|B_k|$ is small. In such cases, we allow a constant factor increase in the upper bound of $|B_{k+1}|$. Conversely, \cref{cl:Bkissmall} handles larger $|B_k|$ with a tighter concentration and increase factor close to 1. 

    \begin{restatable}{claim}{initBksmall}
        \label{cl:initBksmall}
        For any constant $\beta \geq 0$, for $n$ large enough and for any $k \in \NN$ if $|N^{k-1}(u,v)| \leq n/\Delta^2$ then 
        $$|B_{k-1}| \leq \beta |N^{k-1}(u,v)| \implies |B_{k}| \leq \left(3 \beta + \frac{1}{26}\right) |N^{k}(u,v)|$$ with probability at least $1 - n^{-4}$. 
    \end{restatable}

    \begin{restatable}{claim}{Bkissmall}
        \label{cl:Bkissmall}
       For $n$ large enough, for any $k \geq 1$, if $|N^{k-1}(u,v)| \leq n/\Delta^2$  and $|B_{k-1}| \geq \Delta^2/4$ then
        $$|B_{k-1}| \leq  \frac{1}{2} \left(1 + \frac{30}{\log n}\right)^{k-1}|N^{k-1}(u,v)| \implies |B_{k}| \leq  \frac12 \left(1 + \frac{30}{\log n}\right)^{k}|N^{k}(u,v)|$$ with probability at least $1 - n^{-4}$. 
    \end{restatable}

    The $(1 + \frac{30}{\log n})^k$ factor is just a product of the induction. It is important to note that the range of $k$ of interest in this claim is upper bounded by the diameter of $G$, which means $k = O(\log n / \log \log n)$. In this range, we have $(1 + \frac{30}{\log n})^k \leq e^{30k/\log n} = e^{o(1)} = 1 + o(1)$.

    \medskip

   Let us first prove that the lemma follows from the three claims. We suppose that \cref{cl:Akisgood}, \cref{cl:initBksmall} and \cref{cl:Bkissmall} hold deterministically. We will now show that for any $k$ such that $|N^{k-1}(u,v)| \leq n/\Delta^2$, we have $|A_{k}| \geq (1/2 - o(1)) |N^{k-1}(u,v)|$. In particular, we only have to show that $|B_k| \leq (1/2 + o(1) |N^{k-1}(u,v)|$. Therefore, we can consider $i \in \NN$ the first step such that $B_{i+1}$ is non-empty. Note that $i\geq 1$ as $B_1 = \emptyset$. By \cref{cl:initBksmall} applied with $\beta = 0$, we have that $1 \leq |B_{i+1}| \leq \frac1{26}|N^{i+1}(u,v)|$. With a similar reasoning (up to rescaling $\beta$), we obtain the following,
    $$
    |B_{i+2}| \leq \frac2{13}|N^{i+2}(u,v)|\quad \text{and}  \quad |B_{i+3}| \leq \frac12|N^{i+3}(u,v)|.
    $$
    On the lower bound side, note that, we can apply two Chernoff's bounds (\cref{chernoff}). First to lower bound $|B^{k+2}| \leq |\mathcal{B}^{k+2}_1|$, with parameter $\delta = \frac12$ and $\mu = (1 - o(1))\Delta |B^{k+1}| \geq (1 - o(1))\Delta $. Then, to bound $|B^{k+3}| \leq |\mathcal{B}^{k+3}_1|$, with parameter $\delta = \frac12$ and $\mu \geq (1 - o(1))\frac{\Delta^2}{2}|B_{k-1}|$. We obtain,
    $$
        \Pr(|B^{k+3}| \leq \Delta^2/4) \leq 2e^{-\Delta/12} = o(n^{-3}). 
    $$
    
    For $k \geq i+3$, we can apply \cref{cl:Bkissmall} for $\ell$ defined as the largest $k$ such that $|N^{k-1}(u,v)| \leq n/\Delta^2$ is satisfied. Note that as $\frac{2000\log n}{n} \leq p \leq n^{-1/2 - \varepsilon}$, $\ell$ is well-defined. $\ell$ is upper bounded by the diameter of $G$ therefore $\ell \leq \log n /\log\log n$ whp. We obtain the following,
    \begin{align*}
        |A_\ell| &\geq \frac12\left(1 + \frac{30}{\log n}\right)^\ell |N^{\ell}(u,v)| \\
        & \geq \frac12\left(1 + \frac{30}{\log n}\right)^{\frac{\log n}{\log \log n}} \frac{n}{\Delta^2} \\
        & \geq \left(\frac12 + o(1) \right) \frac{n}{\Delta^2}.
    \end{align*}
    It is now sufficient to apply \cref{cl:Akisgood} for $k = \ell$ to conclude that  $|W_{uv}| \geq |A_k|/2 \geq (1/2 - o(1)) n/\Delta^2$ with probability $1 - o(n^{-3})$. By using a union bound over all such pairs. We obtain that all pairs $uv \notin E(G)$, such that $d(u,v)\geq 3$ satisfy $|W_{uv}| \geq (1/2 +o(1)) n/\Delta^2$. Finally, using \cref{cl:dist3} we conclude that $|\mathcal{E}| \geq {n \choose 2} - O(n\Delta^2)$ with probability $1 - o(n^{-1})$.

    \medskip 
    
    We will now prove \cref{cl:Akisgood} - \cref{cl:Bkissmall}. We start by \cref{cl:Bkissmall} as it is the cornerstone of the proof.

    \Bkissmall*

    \begin{proof}[Proof of \cref{cl:Bkissmall}]
    
        We will upper bound the size of the three different parts $\mathcal{B}^k_1, \mathcal{B}^k_2$ and $\mathcal{B}^k_3$ that compose $B_k$ independently:
        
        \begin{itemize}
            \item $\mathcal{B}^k_1 = N(B_{k-1})$: Consider a vertex $x \notin N^{k}(u,v)$, the probability that $x \in N(B_{k-1})$ can be written as: 
            \begin{align*}
                \Pr(x \in N(B_{k-1})) &= \Pr(\operatorname{Bin}(|B_{k-1}|,p) \geq 1) \\
                &= 1 - (1-p)^{|B_{k-1}|} \\
                &= (1 - o(1)) p|B_{k-1}|
            \end{align*}
            therefore, 
            $$\Ex[|\mathcal{B}^k_1|] = (1-o(1))p|B_{k-1}|(n - |B_{k-1}|) = (1-o(1))\Delta|B_{k-1}|\leq \Delta |B_{k-1}|.$$
            By Chernoff's bound (\cref{chernoff}) with $\delta = \frac1{\log n}$, we obtain that 
                \begin{align*}
                    \Pr\left( |\mathcal{B}^k_1| \geq \left(1 + \frac1{\log n} \right) \Delta |B_{k-1}|\right) & \leq e^{- (1 - o(1)) \Delta |B_{k-1}|/(3 \log^2 n) }\\
                    &\leq e^{-(1 - o(1)) \Delta^3/(12\log ^2n)} \\
                    & \leq n^{-5}
                \end{align*}
            for $n$ large enough.

            \item $\mathcal{B}^k_2 = N(\{x \in A_{k-1} \mid N(x) \cap N^{k-1 }(u,v) \neq \emptyset \})$: This set is contained in the neighbourhood of non-isolated vertices in the graph $H := G[N^{k-1}(u,v)]$. Note that the vertex set $N^k(u,v)$ is defined independently of the distribution of edges in $H$. Therefore $H \sim G(|N^{k-1}(u,v)|,p)$. Let $N$ denote $|N^{k-1}(u,v)|$. By assumption we have $\Delta^{2}/4 \leq |B_{k-1}| \leq N \leq n/\Delta^2$ therefore $p = \frac{\Delta}{n} \leq \frac1{N\Delta}$. We can apply \cref{lem:isolatedgnp} to $H$ and deduce that $|\{x \in A_{k-1} \mid N(x) \cup N^{k-1 }(u,v) \neq \emptyset \}| \leq 4N/\Delta$ with probability at least 
            $$1 - 2e^{-N/3\Delta} \geq 1 - 2e^{-\Delta/12} \geq 1 - n^{-5}.$$
            Applying \cref{lem:maxmindegree} we obtain,
                $$\Pr \left( |\mathcal{B}^k_2| \geq 6|N^{k-1}(u,v)| \right) \leq n^{-5}.
                $$
            
            \item $\mathcal{B}^k_3 = \left\{x \in N^k(u,v) \middle| |N(x) \cap A_{k-1}| \geq 2\right\}$: Given a vertex $x \in N^k(u,v)$, consider $M_x$ the event: $|N(x) \cap A_{k-1}| \geq 2$. If $x \in N^k(u,v)$ has a neighbour in $B_{k-1}$ then $x$ has already been counted in $\mathcal{B}^k_1$. Therefore we can suppose $x$ has a neighbour in $A_{k-1}$. The probability for a such a $x$ to satisfy $M_x$ can be written as:
            \begin{align*}
            \Pr( M_x \mid |N(x) \cap A_{k-1}| \geq 1 ) &=  \Pr(|N(x) \cap A_{k-1}| \geq 2 \mid |N(x) \cap A_{k-1}| \geq 1 ) \\
            &= \Pr( \operatorname{Bin}(|A_{k-1}|-1,p) \geq 1) \\
            & \leq 1 - (1-p)^{|A_{k-1}|-1} \\
            & \leq 1 - \left(1 - \frac{\Delta}{n}\right)^{n/\Delta^2}\\
            & \leq 1/\Delta
            \end{align*}
        Note that the $M_x$ are mutually independent, therefore using Chernoff's bounds (\cref{chernoff}), over $|N^k(u,v)| \geq \Delta^2/4$ i.i.d variables with $\delta = 1$ and $\mu \leq |N^k(u,v)|/\Delta$, we deduce,
        $$\Pr \left( |\mathcal{B}^k_3| \geq 2|N^{k}(u,v)|/\Delta \right) \leq n^{-5}.$$
    \end{itemize}
    
    Summing and using a union bound on the three events described above we get that,

        $$\Pr\left(|B_{k}| \leq \left(1 + \frac1{\log n}\right) \Delta|B_{k-1}| + 6|N^{k-1}(u,v)| + \frac{2|N^k(u,v)|}{\Delta}\right) \geq 1 - 3n^{-5}.
        $$

        Let us rewrite this inequality to fit the statement of the theorem. 
        \begin{align*}
            |B_{k}| &\leq \left(1 + \frac1{\log n}\right) \Delta|B_{k-1}| +6|N^{k-1}(u,v)| + \frac{2|N^k(u,v)|}{\Delta}\\
            &\leq \left(1 + \frac1{\log n}\right) \left(1 + \frac{30}{\log n}\right)^{k-1} \frac\Delta2 |N^{k-1}(u,v)| + 6|N^{k-1}(u,v)| + \frac{2|N^k(u,v)|}{\Delta}\\
            &\leq \left(1 + \frac1{\log n} +  \frac{12}{\Delta} \right) \left(1 + \frac{30}{\log n}\right)^{k-1} \frac\Delta2 |N^{k-1}(u,v)| + \frac{2|N^k(u,v)|}{\Delta}\\
            & \leq \left[\left(1 + \frac{2}{\log n} + \frac{12}{\Delta}\right)\left(1 - \frac1{\log n}\right)^{-1} + \frac2\Delta \right] \left(1 + \frac{30}{\log n}\right)^{k-1} \frac12 |N^{k}(u,v)| \\
            & \leq \left( 1 + \frac{30}{\log n}\right) \left(1 + \frac{30}{\log n}\right)^{k-1} \frac12 |N^k(u,v)| \\
            & \leq \frac12 \left(1 + \frac{30}{\log n}\right)^{k} |N^{k}(u,v)|
        \end{align*}

        where we use the claim hypothesis to go from line 1 to 2. From line 3 to 4, we use the following fact, which can be derived from a  Chernoff's bound in the same way as the above upper bound on $N(B_{k-1})$:  $|N^k(u,v)| \geq (1 - \frac1{\log n}) \Delta |N^{k-1}(u,v)|$.
    \end{proof}

    Let us now discuss our bound on $B_k$ for the first few steps. The proof guideline is the same as \cref{cl:Bkissmall}.

    \initBksmall*

    \begin{proof}[Proof of \cref{cl:initBksmall}]
        
        What differs from \cref{cl:Bkissmall} is only that Chernoff's inequality yields a weaker concentration. But the statement is tailored to absorb this lost concentration into the error factor $(3 \beta + \frac1{100})$. 
        \begin{itemize}
            \item  $\mathcal{B}^k_1 = N(B_{k-1})$: We can consider $B_{k-1}$ to be non-empty (otherwise $N(B_{k-1}) = \emptyset$). Again, a Chernoff bound (\cref{chernoff}) with $\delta = 1$ and $\mu = (1 + o(1))\Delta|B_{k-1}|$ implies that,
            \begin{align*}
                \Pr\left(|\mathcal{B}^k_1| \geq 2 \Delta |B_{k-1}|\right) &\leq e^{(1 + o(1)) \Delta/3} \leq n^{-5}.
            \end{align*}
            Note that $|N^{k}(u,v)| \leq \frac32 |N^{k-1}(u,v)|$ by \cref{lem:maxmindegree} therefore,
            $$|\mathcal{B}^k_1| \leq 2 \Delta|B_{k-1}| \leq 2 \Delta \beta |N^{k-1}(u,v)| \leq 3 \beta |N^k(u,v)|.$$
            
            \item $\mathcal{B}^k_2 = N(\{x \in A_{k-1} \mid N(x) \cap N^{k-1}(u,v) \neq \emptyset \})$: Instead of directly bounding the size of $\mathcal{B}^k_2$ we use the fact that $|\mathcal{B}^k_2|$ is upper bounded by two times the number of edges in the graph $G[N^{k-1}(u,v)]$. Let us denote this random variable by $M :=|E(G[N^{k-1}(u,v)])|$ and $\mu = \Ex[M]$. Note that $\mu = O(p|N^{k-1}(u,v)|^2) \ll |N^{k-1}(u,v)|$ by assumption. Therefore we can fix $\delta_2 \geq 1$ to satisfy $(1 + \delta_2)\mu = \frac{1}{104}|N^ {k-1}(u,v)|$ and apply Chernoff's bounds
            \begin{align*}
                \Pr\left( |\mathcal{B}^k_2| \geq \frac{\Delta}{52} |N^{k-1}(u,v)|\right)  \leq \Pr\left( M \geq \frac{1}{104} |N^{k-1}(u,v)|\right) &\leq e^{-(1 + o(1))\delta_2 \mu/3} \\
                &\leq e^{-\frac{1+o(1)}{309}|N^1(u,v)|} \\
                &\leq n^{-5}.
            \end{align*}
        \item $\mathcal{B}^k_3 = \{x \in N^k(u,v) \mid |N(x) \cap A_{k-1}| \geq 2\}$. From the same reasoning as \cref{cl:Bkissmall} proof, we know that $\Pr(|N(x) \cap A_{k-1}| \geq 2 \mid x \in A_k) \leq 1/\Delta$. We can also use Chernoff's bounds (\cref{chernoff}). Fix $\delta_3 \geq 1$ which satisfy $(1+\delta_3)\Ex(|\mathcal{B}^k_2|) =  \frac{\Delta}{104}|N^1(u,v)| \leq \frac1{52}|N^2(u,v)|$. As for $\mathcal{B}^k_2$ we have,
        \begin{align*}
            \Pr(|\mathcal{B}^k_3| \geq  \frac{\Delta}{104}|N^1(u,v)| ) &\leq e^{-\Delta^2/309} \\
            &\leq n^{-5}.
        \end{align*}
        \end{itemize}
       
    Using a union bound on the three events described above, we obtain:

    \begin{align*}
    \Pr\left(|B_{k}| \leq 3 \beta |N_{k}(u,v)| + \frac{2|N^{k}(u,v)|}{52}\right) \geq 1 - 3n^{-5}.
    \end{align*}
    \end{proof}

Finally, we are only left to prove that $A_k$ contains indeed a large number of witnesses of the pair $u,v$.

\Akisgood*

\begin{proof}[Proof of \cref{cl:Akisgood}]
        We will prove that any vertex $x \in A_k$ such that $N(x) \cap N^k(u,v) = \emptyset$ is a witness of $u,v$. We already proved that most of the vertices satisfy this condition when we upper bounded $|\{x \in A_{k-1}| N(x) \cap N^{k-1}(u,v) \neq \emptyset\}|$ in \cref{cl:initBksmall}. 
        
        We prove by induction on $k$ that if $N(x) \cap N^k(u,v) = \emptyset$ then $x$ is a witness of $u,v$. It is true for $k = 1$ as $d(u,v) \geq 3$. Assume that the statement is true up to $k-1$. Consider the unique $y \in A_{k-1}$ such that $xy \in E(G)$. This must exist by construction of $A_{k}$. Note that the definition $\mathcal{B}^k_2$ ensures that if $x \in A_k$ then $y$ satisfies the induction hypothesis: $N(y) \cap N^{k-1}(u,v) = \emptyset$.  Therefore, we can consider up to symmetry that $d(y,u) = k-1$ and $d(y,v)\geq k+1$. Thus $d(x,u) = k$ and $d(y,v) \geq k+2$ as it unique neighbour in $N^{k-1}(u,v) \cup N^k(u,v)$ is $y$. Therefore, $x$ is a witness of $u,v$.
    \end{proof}
\end{proof}

\main*

\begin{proof}
    Recall that $\Delta = (n-1)p$. We will prove an upper bound on the expected query complexity of $O(\Delta^2 n \log n)$. Let $\overline{E} := {V(G) \choose 2} \setminus E(G)$ be the set of non-edges of $G$. The algorithm proceeds as follows. First, we consider a randomly sampled set $S$ of vertices where $|S| = 3\Delta^2 \log n$. We ask $\qu(V,S)$. We then deduce
    $$D = \{ \{u,v\} \in {V(G) \choose 2} \mid \exists s \in S, |d(s,u) - d(s,v)| \geq 2\},$$ and query all pairs in ${V(G) \choose 2} \setminus D$. Note that $D \subseteq \overline{E}$ by definition, therefore the correctness of this algorithm is direct.
    
    We now prove that ${[n] \choose 2} \setminus D = O(\Delta^2 n)$ with probability $1-o(n^{-1})$. Consider the set $D = \{uv \in {[n] \choose 2} \mid |W_{uv}| \geq n/\Delta^2\}$. \cref{lem:linearwitnesses} ensures us that $|D| = {n \choose 2} - O(\Delta^2n)$ with probability $1 - o(n^{-1})$. Assume that \cref{lem:linearwitnesses} holds deterministically. We compute the probability that a fixed pair $uv \in D$ has no witness in $S$.
    
    \begin{align*}
        \Pr(S \cap W_{uv} = \emptyset) & \leq  \left(1 - \frac{n}{n\Delta^2}\right)^{3 \Delta^2 \log (n)}\\
        & \leq e^{-3\log (n)}\\
        & \leq o(n^{-3})
    \end{align*}

    If we union bound over the events $S\cap W_{uv} \neq \emptyset$ for all pairs of vertices, we conclude that $|{[n] \choose 2} \setminus \Tilde{E}| = O(\Delta^2n)$ with probability $1 - o(n^{-1})$, in which case, our algorithm uses $n|S| + O(\Delta^2n) = O( \Delta^2 n \log n)$ queries. The worst-case query complexity of the algorithm being $O(n^{2})$ we obtain that the expected number of queries done by the algorithm is $(1 - o(n^{-1}))\Delta^2 n\log n + o(n^{-1}) n^2 = O(\Delta^2 n \log n).$
\end{proof}

\section{Acknowledgment}

The author extends their gratitude to Zoé Varin and Carla Groenland for the fruitful discussions on the topic. They also thank Julien Duron for his comments on the first version of this paper.

\bibliographystyle{alpha}
\newcommand{\etalchar}[1]{$^{#1}$}

\end{document}